\documentclass[12pt, reqno]{amsart}

\usepackage[body={6.3in,9.3in}]{geometry}
\usepackage{verbatim}
\usepackage[ansinew]{inputenc}
\usepackage{textcomp}
\usepackage{latexsym}
\usepackage[usenames]{color}
\usepackage{amsmath,amsfonts,amssymb,amsthm}
\usepackage{graphicx}
\usepackage{longtable}
\usepackage{bbm}
\usepackage{setspace}
\usepackage{array}

\def\qed{\hfill{\raggedleft{\hbox{$\Box$}}} \smallskip}

\def\R{\mathbb{R}}

\DeclareMathOperator*{\argmax}{argmax}

\newcommand{\ds}{\displaystyle}

\theoremstyle{plain} \newtheorem{lem}{Lemma}
\theoremstyle{plain} \newtheorem{prop}[lem]{Proposition}
\theoremstyle{plain} \newtheorem{thm}[lem]{Theorem}
\theoremstyle{plain} 
\theoremstyle{plain} 
\theoremstyle{definition} 
\theoremstyle{definition}
\theoremstyle{definition} 
\theoremstyle{definition} 
\theoremstyle{definition}\newtheorem{ex}[lem]{Example}

\newlength\savedwidth

\begin{document}

\title{Combinatorial Types of Tropical Eigenvectors}
\author{Bernd Sturmfels and Ngoc Mai Tran}
\address{ Department of Statistics, University of California, Berkeley, CA 94720-3860, USA}
\email{bernd@math.berkeley.edu, tran@stat.berkeley.edu}
\urladdr{www.math.berkeley.edu/~bernd/, www.stat.berkeley.edu/~tran/}
\thanks{Both authors
were supported by the U.S.~National Science Foundation
(DMS-0757207 and DMS-0968882)}

\begin{abstract}
The map which takes a square matrix to its tropical
eigenvalue-eigenvector pair is piecewise linear.
We determine the cones of linearity of this map. They are simplicial but
they do not form a fan. Motivated by statistical ranking,
 we also study  the restriction of
that cone decomposition to the
subspace of skew-symmetric matrices.
\end{abstract}

\maketitle

\section{Introduction}
Applications such as discrete event systems \cite{BCOQ} lead to the
tropical eigenvalue equation
\begin{equation}
\label{eq:eigeig} A \odot x \,=\, \lambda \odot x .
\end{equation}
Here arithmetic takes place in the
{\em max-plus algebra} $(\mathbb{R}, \oplus, \odot)$, defined by
$u\oplus v = \max{\{u,v\}}$ and $ u\odot v = u + v$.
The real  $n {\times} n$-matrix $A = (a_{ij})$ is fixed. One seeks
to compute all {\em tropical eigenpairs} $(\lambda, x) \in \R \times \R^n$,
that is, solutions of (\ref{eq:eigeig}).
If $(\lambda,x)$ is such a pair for $A$ then so is
$(\lambda, \nu \odot x)$ for any $\nu \in \R$.  We regard
these eigenpairs as equivalent. The pairs $(\lambda,x)$ are thus viewed as elements in
$\R \times \mathbb{TP}^{n-1}$ where
$\mathbb{TP}^{n-1} = \R^n/\R(1,1,\ldots1)$  is the {\em tropical projective torus} \cite{bernd.trop}.
Our point of departure will be the following result.

\begin{prop}
\label{thm:eins}
There exists a partition of matrix space $\R^{n \times n}$ into 
finitely many convex polyhedral cones such that 
each matrix in the interior of a full-dimensional
cone has a unique eigenpair $(\lambda, x)$ in $\R \times \mathbb{TP}^{n-1}$.
Moreover, on each full-dimensional cone in that partition,
the {\em eigenpair map} $A \mapsto (\lambda(A),x(A))$ is represented by a unique linear function
$\,\R^{n \times n} \rightarrow \R \times \mathbb{TP}^{n-1}$.
\end{prop}

In  tropical linear algebra \cite{butkovic} it is known that the eigenvalue is unique, but the projective tropical eigenspace can be of dimension anywhere between 0 and $n-1$. The proposition implies that the set of matrices with more than one eigenvector lies in the finite union of subspaces of codimension one, and hence a 
generic $n \times n$ matrix has a unique eigenpair. 

The eigenvalue $\lambda(A)$ is the maximum cycle mean of the weighted directed graph with edge weight matrix
$A$; see \cite{BCOQ, butkovic, CG}. As we shall see in (\ref{dual_LP}) below,
the map $A \mapsto \lambda(A)$ is the support function of a convex polytope,
 and hence it is piecewise linear. In this article we study the refinement from eigenvalues to eigenvectors.
 Our main result is as follows:
 
\begin{thm}\label{thm:zwei}
The open cones in $\R^{n \times n}$ on which the eigenpair map
is represented by distinct and unique linear functions
are all linearly isomorphic to $\,\R^n \times \R_{> 0}^{n(n-1)}$. 
These cones are indexed by the connected functions  $\, \phi:[n] \rightarrow [n]$, so their number~is
\begin{equation}
\label{eq:oeisnumber} \sum_{k=1}^n \frac{n!}{(n-k)!} \cdot n^{n-k-1}. 
\end{equation}
For $n \geq 3$, these cones do not form a fan, that is,
two cones may intersect in a non-face.
\end{thm}

Here a function $\phi$ from $[n] = \{1,2,\ldots,n\}$ to itself  is called {\em connected}
if its graph is connected as an undirected graph.
The count in (\ref{eq:oeisnumber}) is the sequence A001865 in \cite{oeis}.
In Section 2, we explain this combinatorial representation and we prove both
Proposition \ref{thm:eins} and Theorem \ref{thm:zwei}.
For $n=3$, the number (\ref{eq:oeisnumber}) equals $17$,
and our cone decomposition is represented by a
$5$-dimensional simplicial complex with f-vector 
$(9,36, 81, 102, 66, 17) $. 
The locus in $\R^{n \times n}$ where the cone decomposition fails to be a fan
consists precisely of the matrices $A$ whose eigenspace
is positive-dimensional.
We explain the details  in Section \ref{sec:example}.

 In Section \ref{sec:symm} we restrict to matrices $A$ that are skew-symmetric, in symbols: $A = -A^T$.  
 Tropical eigenvectors of skew-symmetric matrices arise in  {\em pairwise comparison ranking},
  in the approach that was pioneered by
    Elsner and van den Driessche \cite{elsner, elsner10}. In \cite{nmt}, the second author
    offered a comparison with two other methods for statistical ranking, and
        she noted that the eigenvalue map $A \mapsto \lambda(A)$
    for skew-symmetric $A$ is linear on (the cones over) the facets of the
      cographic zonotope associated with the complete graph on $n$ vertices.
  The tropical eigenvector causes a further subdivision for many of the facets,
  as seen for $n=4$ in \cite[Figure 1]{nmt}. Our Theorem \ref{thm:drei}
 characterizes these subdivisions into cubes for all~$n$.
We close with a brief discussion of the eigenspaces of non-generic matrices.
   
\section{Tropical Eigenvalues and Eigenvectors}

We first review the basics concerning tropical eigenvalues and eigenvectors,
and we then prove our two results. Let $A$ be a real $n \times n$-matrix
and $G(A)$ the corresponding weighted directed graph on $n$ vertices.
It is known that $A$ has
a unique tropical eigenvalue $\lambda(A)$. This eigenvalue
can be computed as the optimal value of
the following linear program:
\begin{equation}
\label{primal_LP}
 {\rm Minimize}\, \,\,\lambda \,\,\, \hbox{subject to} \quad
 a_{ij} + x_j \leq \lambda + x_i\,\,\,
\hbox{for all $\,1 \leq i,j \leq n$}.
\end{equation}
Cuninghame-Green \cite{CG} used the formulation (\ref{primal_LP}) 
to show that the eigenvalue $\lambda(A)$ of a matrix $A$ can be computed
in polynomial time.
For an alternative approach to the same problem we refer to Karp's article \cite{karp}.
The linear program dual to (\ref{primal_LP}) takes the form
\begin{equation}
\label{dual_LP}
\begin{matrix}
& {\rm Maximize} \,\, \sum_{i,j=1}^n a_{ij} p_{ij}
\,\, \,\, \hbox{subject to} \,\,\,\, p_{ij} \geq 0
\,\,\hbox{ for } 1 \leq i,j \leq n , \\ &
 \sum_{i,j=1}^n  p_{ij} = 1  \,\,\,\, \hbox{and}\,\,\,
\sum_{j=1}^n p_{ij} = \sum_{k=1}^n p_{ki}
\,\,\hbox{ for all }\, 1 \leq i \leq n.
\end{matrix}
\end{equation}
The $p_{ij}$ are the variables, and
the constraints require $(p_{ij})$ to be a probability distribution on the edges of $G(A)$ that represents a flow in the directed graph.

 Let $\mathcal{C}_n$ denote the 
$n(n-1)$-dimensional convex
polytope of all feasible solutions to (\ref{dual_LP}). By strong duality, the primal (\ref{primal_LP})
and the dual (\ref{dual_LP}) have the same optimal value.
This implies that the eigenvalue function
$A \mapsto \lambda(A)$ is the
support function of the convex polytope $\mathcal{C}_n$. 
Hence the function $A \mapsto \lambda(A)$ is continuous, convex and piecewise-linear.

By the {\em eigenvalue type} of a matrix $A \in \R^{n \times n}$ we shall mean
the cone in the {\em normal fan} of the polytope $\mathcal{C}_n$
that contains~$A$. Since each vertex of $\mathcal{C}_n$ is the
uniform probability distribution on a directed
cycle in $G(A)$, the eigenvalue $\lambda(A)$ is the maximum cycle mean of $G(A)$. Thus,
the open cones in the normal fan of $\mathcal{C}_n$ are naturally indexed by cycles in the graph on $n$ vertices. 
The cycles corresponding to the normal cone containing the matrix $A$ are the \textit{critical cycles} of $A$.
The union of their vertices is called the set of \textit{critical vertices} in~\cite{butkovic,elsner}

\begin{ex}
\label{ex:bipyramid}
Let $n=3$. There are eight cycles,
two of length $3$, three of length $2$ and three of length $1$,
and hence eight eigenvalue types.
The polytope $\mathcal{C}_3$ is six-dimensional:
it is the threefold pyramid over the bipyramid 
formed by the $3$-cycles and $2$-cycles. \qed
\end{ex}

We have seen that  the normal fan of $\mathcal{C}_n$ partitions $\mathbb{R}^{n\times n}$ into polyhedral cones 
on which of the eigenvalue map $A \mapsto \lambda(A)$ is linear. Our goal is to refine the normal fan of $\mathcal{C}_n$ into cones of linearity for the eigenvector $A \mapsto x(A)$ map. To prove our first result, we introduce some notation and recall some properties of the tropical eigenvector. For a matrix $A \in \mathbb{R}^{n\times n}$, let $B := A \odot (-\lambda(A))$. For a path $P_{ii'}$ from $i$ to $i'$, let $B(P_{ii'})$ denote 
its length (= sum of all edge weights along the path) in the graph of $B$. 
We write
\[ \{\Gamma_{ii'}\} \,:=\, \ds \argmax_{P_{ii'}} B(P_{ii'})\] 
for the set of paths of maximum length from $i$ to $i'$, and write $\Gamma_{ii'}$ if the path is unique. Note that $B(\Gamma_{ii'})$ is well-defined even if there is more than one maximal path, and it is finite since all cycles of $B$ are non-positive. If $j, j'$ are intermediate vertices on a path $P_{ii'}$, then $P_{ii'}(j \to j')$ is the path from $j$ to $j'$ within~$P_{ii'}$. 

It is known from tropical linear algebra \cite{bapat,butkovic}
 that the tropical eigenvector $x(A)$ of a matrix $A$ 
 is unique if and only if the union of its  critical cycles is connected.
 In such cases, the eigenvector $x(A)$ can be calculated by first fixing a critical vertex $\ell$, 
 and then setting
\begin{equation} x(A)_i \,\,= \,\, B(\Gamma_{i\ell}), \label{eqn:evec} \end{equation}
that is, the entry $x(A)_i$ is the maximal length among paths from $i$ to $\ell$ in the graph of $B$.

\begin{proof}[Proof of Proposition \ref{thm:eins}]
Following the preceding discussion, it is sufficient to construct the refinement of each eigenvalue type in the normal fan of $\mathcal{C}_n$. Let $A$ lie in the interior of such a cone. Fix a critical vertex $\ell$. Since the eigenvalue map is linear, for any path $P_{i\ell}$
the quantity $B(P_{i\ell})$ is given by a unique linear form in the entries of $A$.
 A path $Q_{i\ell}$ is maximal if and only if $B(Q_{i\ell}) - B(P_{i\ell}) \geq 0$ for all paths $P_{i\ell} \neq Q_{i\ell}$.
 Hence, by (\ref{eqn:evec}), 
 the coordinate $x(A)_i$ of the eigenvector
 is given by a unique linear function in the entries of $A$ (up to choices of $\ell$) if and only if $\{\Gamma_{i\ell}\}$ has cardinality one, or, equivalently, if and only if 
\begin{equation}
B(Q_{i\ell}) - B(P_{i\ell}) > 0 \,\,\, \mbox{  for all paths  }\, P_{i\ell} \neq Q_{i\ell}. \label{eqn:path.diff}
\end{equation}
 We now claim that, as linear functions in the entries of $A$, the linear forms in (\ref{eqn:path.diff}) are independent of the choice of $\ell$.
Fix another critical vertex $k$. It is sufficient to prove the claim when $(\ell \to k)$ is a critical edge. In this case, for any path $P_{i\ell}$, the path $R_{ik} := P_{i\ell} + (\ell \to k)$ is a path from $i$ to $k$ with $B(R_{ik}) = B(P_{i\ell}) + a_{\ell k} - \lambda(A)$. Conversely, for any path $R_{ik}$, traversing the rest of the cycle from $k$ back to $\ell$ gives a path $P_{i \ell} := R_{ik} + (k \to \ldots \to \ell)$ from $i$ to $\ell$, with $B(P_{i\ell}) = B(R_{ik}) - (a_{\ell k} - \lambda(A))$, since the critical cycle has length $0$ in the graph of $B$. Hence, the map $B(P_{i\ell}) \mapsto B(P_{i\ell}) + a_{\ell k} - \lambda(A)$ is a bijection taking the lengths of paths from $i$ to $\ell$ to the lengths of paths from $i$ to $k$. Since this map is a tropical scaling, the linear forms in (\ref{eqn:path.diff}) are unchanged, and hence they are independent of the choice of $\ell$.
We conclude that  (\ref{eqn:path.diff}) defines the cones promised in Proposition \ref{thm:eins}. 
\end{proof}

Two points should be noted in the proof of Proposition \ref{thm:eins}. Firstly, in the interior of each 
eigenpair cone  (\ref{eqn:path.diff}), for any fixed critical vertex $\ell$ and any 
other vertex $i \in [n]$, the maximal path $\Gamma_{i\ell}$ is unique. Secondly, the number of facet defining equations for these cones are potentially as large as the number of distinct paths from $i$ to $\ell$ for each $i \in [n]$. In Theorem \ref{thm:zwei} we shall show that there are only $n^2-n$ facets. Our proof relies on the following lemma, which is based on an argument we learned from \cite[Lemma 4.4.2]{butkovic}.  

\begin{lem}\label{lem:i.star}
Fix $A$ in the interior of an eigenpair cone (\ref{eqn:path.diff}).
For each non-critical vertex $i$, there is a unique critical vertex $i^\ast$ 
such that the path $\Gamma_{ii^\ast}$ uses no edge in the critical cycle. 
If $j$ is any other non-critical vertex on the path $\Gamma_{ii^\ast}$, then $j^\ast = i^\ast$ 
and $\Gamma_{jj^\ast} = \Gamma_{ii^\ast}(j \to i^\ast)$. 
\end{lem}

\begin{proof}
We relabel vertices so that the critical cycle is $(1 \to 2 \to \ldots \to k \to 1)$.
For any non-critical $i$ and critical $\ell$, the path
$\Gamma_{i\ell}$ is unique, and by the same argument as in the proof of Proposition \ref{thm:eins}, $\Gamma_{i(\ell+1)} = \Gamma_{i\ell} + (\ell \to (\ell+1))$ . Hence there exists a unique critical vertex $i^\ast$ such that $\Gamma_{ii^\ast}$ uses no edge in the critical cycle. 

For the second statement, we note that  $\Gamma_{ii^\ast}(j \to i^\ast)$ uses no edge in the critical cycle. Suppose 
 that $\Gamma_{ji^\ast} \neq \Gamma_{ii^\ast}(j \to i^\ast)$. The concatenation of
 $\Gamma_{ii^\ast}(i \to j) $ and $ \Gamma_{ji^\ast}$ is a path from $i$ to $i^\ast$
 that is longer than  $\Gamma_{ii^\ast}$.
 This is a contradiction and the proof is complete.
 \end{proof}

\begin{proof}[Proof of Theorem \ref{thm:zwei}] 
We define the  {\em critical graph} of $A$ to be the subgraph of
$G(A)$ consisting of all edges in the critical cycle
and all edges in the special paths $\Gamma_{i i^*}$ above.
Lemma \ref{lem:i.star} says that the critical  graph is
the union of the critical cycle with trees rooted at the critical vertices. 
Each tree is directed towards its root.
Hence the critical graph is a connected function $\phi$ on $[n]$,
and this function $\phi$ determines the eigenpair type of the matrix $A$.

We next argue that every connected function $\phi : [n] \rightarrow [n]$ is the critical graph
of some generic matrix $A \in \R^{n \times n}$. If $\phi$ is surjective then $\phi$ is a cycle and we
take any matrix $A$ with the corresponding eigenvalue type.
Otherwise, we may assume that $n$ is not in the image of $\phi$. By induction we can 
find an $(n{-}1) \times (n{-}1)$-matrix $A'$ with critical graph
$\phi \backslash \{(n,\phi(n))\}$. We enlarge $A'$ to the desired $n {\times} n$-matrix $A$
by setting $a_{n,\phi(n)} = 0$ and all other entries very negative. Then $A$ has $\phi$ as its critical graph.
We conclude that, for every connected function $\phi$ on $[n]$,
the set of all $n \times n$-matrices that have the critical graph $\phi$ is a
full-dimensional convex polyhedral cone 
$\Omega_\phi$ in $\R^{n\times n}$, and these are the open cones,
characterized in  (\ref{eqn:path.diff}), on which the eigenpair map is linear.

We next show that these cones are linearly isomorphic to $\,\R^n \times \R_{\geq 0}^{n(n-1)}$. Let $e_{ij}$ denote the standard basis
matrix of $\R^{n \times n}$ which is $1$  in
position $(i,j)$ and $0$ in all other positions. 
Let  $V_n$ denote the $n$-dimensional linear subspace of $\R^{n \times n}$
spanned by the matrices $\sum_{i,j=1}^n  e_{ij} $ and 
$\sum_{j=1}^n e_{ij} - \sum_{k=1}^n e_{ki}$ for $i=1,2,\ldots,n$. Equivalently, $V_n$ is the orthogonal complement to the affine span of 
the cycle polytope $\mathcal{C}_n$. The normal cone at each vertex of $\mathcal{C}_n$ is the sum of $V_n$ and a pointed cone of dimension $n(n-1)$.
We claim that  the subcones $\Omega_\phi$ inherit the same property.
Let $\bar \Omega_\phi$ denote the image of $\Omega_\phi$ in the
quotient space $\R^{n \times n}/V_n$. This is an $n(n-1)$-dimensional pointed
convex polyhedral cone, so it has  at least $n(n-1)$ facets. 
To show that it has precisely $n(n-1)$ facets, we claim that
\begin{equation}
\label{eq:OmegaCone3}
\Omega_\phi \quad = \quad \bigl\{ \, A \in \R^{n \times n} \,:\,
b_{ij} \,\leq\, B(\phi_{ij^\ast}) - B(\phi_{jj^\ast}) \,\,: \,(i,j) \in [n]^2 \backslash \phi \,\bigr \}. \qquad \quad 
\end{equation}
In this formula, $\phi_{ii^\ast}$ denotes the directed path from $i$ to $i^\ast$ in the graph of $\phi$, and $B = (b_{ij}) = (a_{ij} - \lambda_\phi(A))$, where $\lambda_\phi(A)$ is the mean of the cycle in the graph of $\phi$ 
with edge weights $(a_{ij})$. The inequality representation (\ref{eq:OmegaCone3}) will imply
 that the cone $\Omega_\phi$ is linearly isomorphic to  $\R^n \times \R_{\geq 0}^{n(n-1)}$ because there are $n(n-1)$ non-edges $(i,j) \in [n]^2\backslash \phi$. 

Let $A$ be any matrix for which the
 $n(n-1)$ inequalities in (\ref{eq:OmegaCone3}) hold strictly for non-edges of $\phi$. Let $\psi$ denote the connected function corresponding to the critical graph of $A$. To prove the claim,
 we must show that $\psi = \phi$. 
 First we show that $\psi$ and $\phi$ have the same cycle.
  Without loss of generality, let $(1 \to2 \to \ldots \to k \to 1)$ be the cycle in $\phi$, and 
$(i_1 \to i_2 \to \ldots \to i_m \to i_1)$ the cycle in $\psi$.
Assuming they are different, the inequality in (\ref{eq:OmegaCone3}) holds 
strictly for at least one edge in $\psi$. 
Using the identities 
$\,B(\phi_{i_j i_{j+1}^*}) =  B(\phi_{i_j i_{j}^*})  + B(\phi_{i_j^* i_{j+1}^*}) $, we find
\begin{eqnarray*}
\!\! && b_{i_1i_2} + b_{i_2i_3} + \cdots + b_{i_mi_1} \\
\!\! &<& B(\phi_{i_1i_2^\ast}) - B(\phi_{i_2i_2^\ast}) \,+\, B(\phi_{i_2i_3^\ast}) - B(\phi_{i_3i_3^\ast})\, + \,\cdots \,+\, B(\phi_{i_mi_1^\ast}) - B(\phi_{i_1i_1^\ast}) \\
\!\! &=& \! B(\phi_{i_1i_1^\ast}) {+} B(\phi_{i_1^\ast i_2^\ast}) {-} B(\phi_{i_2i_2^\ast}) 
+ B(\phi_{i_2i_2^\ast}) {+} B(\phi_{i_2^\ast i_3^\ast}) 
{-}\,\cdots + B(\phi_{i_mi_m^\ast}) {+} B(\phi_{i_m^\ast i_1^\ast}) {-} B(\phi_{i_1i_1^\ast})  \\
\!\! &=& B(\phi_{i_1^\ast i_2^\ast}) + B(\phi_{i_2^\ast i_3^\ast}) + \cdots + B(\phi_{i_m^\ast i_1^\ast})
\quad = \quad 0 \quad = \quad b_{12} + b_{23} + \cdots + b_{k1}.
\end{eqnarray*}
This contradicts that $\psi$ has maximal cycle mean, hence $\psi$ and $\phi$ have the same unique critical cycle. It remains to show that other edges agree. Suppose for contradiction that there exists a non-critical vertex $i$ in which $\psi(i) \neq \phi(i)$. Since $(i, \psi(i))$ is a non-edge in $[n]^2\backslash \phi$, 
the inequality (\ref{eq:OmegaCone3}) holds strictly by the assumption on the choice of $A$, and we get
\[ B(\phi_{i\psi(i)^\ast}) \,>\, B(\phi_{\psi(i)\psi(i)^\ast}) + b_{i\psi(i)} \,=\, B((i\to \psi(i)) + \phi_{\psi(i)\psi(i)^\ast}). \]
This shows that the path $(i\to \psi(i)) + \phi_{\psi(i)\psi(i)^\ast}$  is not critical, that is, it is not in the graph of $\psi$. Hence, there exists another vertex $i_2$ along the path $\phi_{\psi(i)\psi(i)^\ast}$ such that $\psi(i_2) \neq \phi(i_2)$. Proceeding by induction, we obtain a sequence of vertices $\,i, i_2, i_3, \ldots,$
with this property. Hence eventually we obtain a cycle in $\psi$ that consists entirely of non-edges in $[n]^2\backslash\phi$. But this contradicts the earlier statement that the unique critical cycle in $\psi$ agrees with that in $\phi$. This completes the proof of the first sentence in Theorem \ref{thm:zwei}.

For the second sentence we note that
the number of connected functions in (\ref{eq:oeisnumber}) is the sequence A001865 in \cite{oeis}.
Finally, it remains to be seen that our simplicial cones do not form
a fan in $\R^{n^2}/V_n$ for $n \geq 3$.
We shall demonstrate this explicitly in Example~\ref{ex:counter}.
\end{proof}

\section{Eigenpair Cones and Failure of the Fan Property} \label{sec:example}

Let $(x_\phi,\lambda_\phi): \mathbb{R}^{n\times n} \rightarrow  \mathbb{TP}^{n-1}\times\mathbb{R}$ denote the unique linear map which takes any matrix $A$ in the interior of the cone $\Omega_\phi$ to its eigenpair $(x(A),\lambda(A))$. Of course, this 
linear map is defined  on all of $\mathbb{R}^{n \times n}$, not just on $\Omega_\phi$. The following lemma is a useful characterization of $\Omega_\phi$ in terms of the 
linear map $(x_\phi,\lambda_\phi)$
 which elucidates its product structure as $\mathbb{R}^n \times \R^{n(n-1)}_{\geq 0}$.

\begin{lem} \label{lem:decomp}
 For a matrix $A \in \mathbb{R}^{n\times n}$, we abbreviate
  $\,x := x_\phi(A), \,\lambda := \lambda_\phi(A)$,
  and we set $C = (c_{ij}) = (a_{ij} - x_i + x_j- \lambda)$. Then $A$ is in the interior of the cone $\Omega_\phi$ if and only if
\begin{equation}\label{eqn:C} C_{i\phi(i)} \,= \,0 \,\,\mbox{  for all  } \, i \in [n]
\,\,\,\, \hbox{and}  \,\,\,\, C_{ij} < 0 \mbox{  otherwise.  } \end{equation}
\end{lem} 

\begin{proof} Since the matrix $(x_i-x_j+\lambda)$ is in the linear subspace $V_n$, 
the matrices $A$ and $C$ lie in the same eigenpair cone $\Omega_\phi$.
Since $C_{ij} \leq 0$ for all $i, j = 1, \ldots, n$, the conditions (\ref{eqn:C}) are thus equivalent to 
 \[ (C \odot [0,\ldots,0]^T)_i \,=\,  \max_{k\in[n]} C_{ik} \,=\, C_{i\phi(i)} \,=\,  0  \quad \hbox{ for all $i \in [n]$.} \]
In words, the matrix $C$ is a normalized version of $A$
which has eigenvalue $\lambda(C) = 0$ and eigenvector $x(C) = [0,\ldots,0]^T$.
The condition (\ref{eqn:C}) is equivalent to that in (\ref{eq:OmegaCone3}), with strict inequalities for 
$\{(i,j): j\neq \phi(i)\}$, and it holds if and only if $C$ is in the interior of $\Omega_\phi$.
\end{proof}

The linear map $A \mapsto (C_{ij}: j \neq \phi(i))$ defined in Lemma \ref{lem:decomp} 
realizes the projection from the eigenpair cone $\Omega_\phi$ onto 
its pointed version $\bar \Omega_\phi$. 
Thus, the simplicial cone $\bar \Omega_\phi$ is 
spanned by the images in $\R^{n \times n} /V_n$  
of the matrices $-e_{ij}$ that are indexed by the $n(n-1)$ non-edges:
\begin{equation}
\label{eq:OmegaCone1}
\quad \Omega_\phi \quad = \quad
V_n \,\, + \,\, \R_{\geq 0} \bigl\{ -e_{ij} \,:\,(i,j) \in [n]^2 \backslash \phi \bigr\}
\qquad \simeq \quad \R^n \times \R_{\geq 0}^{n(n-1)}.
\end{equation}
At this point, we find it instructive to work out the eigenpair cone
$\Omega_\phi$ explicitly for a small example, and 
to verify the equivalent  representations (\ref{eq:OmegaCone3}) and
(\ref{eq:OmegaCone1}) for that example.

\begin{ex}[$n\!=\!3$] \label{ex:ineq}
Fix the connected function $\phi = \{12,23,31\}$. 
 Its eigenvalue functional is  $\lambda := \lambda_\phi(A) = \frac{1}{3}(a_{12} + a_{23} + a_{31})$.
 The eigenpair cone  $\Omega_\phi$ is $9$-dimensional and  is
  characterized by $3\cdot 2 = 6$ linear inequalities,   one for each of
 the six non-edges $(i,j)$, as in (\ref{eq:OmegaCone3}).
 For instance, consider the non-edge $(i,j) = (1,3)$.
 Using the identities
 $B(\phi_{13^\ast}) = b_{12} + b_{23} = a_{12} + a_{23} - 2 \lambda$
 and $B(\phi_{33^\ast}) = b_{31}+b_{12}+b_{23} =   a_{31}+a_{12}+a_{23} - 3 \lambda$,
  the inequality $ \,b_{13} \leq B(\phi_{13^\ast}) - B(\phi_{33^\ast}) \,$ in
(\ref{eq:OmegaCone3}) translates into $\, a_{13} \,\leq\,    2 \lambda -   a_{31} \,$ and hence~into
\[ a_{13} \,\, \leq \,\, \frac{1}{3}  (2 a_{12} + 2 a_{23} - a_{31}).\]
Similar computations for all six non-edges of $\phi$ give the
following six linear inequalities:
\begin{eqnarray*}
a_{11} \leq \frac{1}{3}( a_{12}+  a_{23} + a_{31})  , &
a_{22} \leq  \frac{1}{3} (a_{12} + a_{23} +  a_{31}) ,&
a_{33} \leq  \frac{1}{3} (a_{12} + a_{23}+ a_{31}) , \\
a_{13} \leq \frac{1}{3}  (2 a_{12} + 2 a_{23} - a_{31}) ,&
a_{32} \leq \frac{1}{3} (2 a_{12} -  a_{23} + 2 a_{31}) ,  &
a_{21} \leq \frac{1}{3} (-a_{12}+ 2 a_{23} + 2 a_{31}).
\end{eqnarray*}
The eigenpair cone $\Omega_\phi$ equals the
set of solutions in $\mathbb{R}^{3 \times 3}$ to this system of inequalities.

According to Lemma \ref{lem:decomp}, these same inequalities
can also derived from (\ref{eq:OmegaCone3}). We have
$$  x: \,= \,x_\phi(A) \,= \, \bigl[\, a_{12} + a_{23} - 2 \lambda, \,a_{23} - \lambda, \,0 \, \bigr]^T. $$
The equations $c_{12} = c_{23} = c_{31} = 0$ 
in (\ref{eq:OmegaCone1}) are equivalent to
\begin{eqnarray*} a_{12} = x_1 - x_2 + \lambda, & a_{23} = x_2 - x_3 + \lambda, & a_{31} = x_3 - x_1 + \lambda,
\end{eqnarray*}
and the constraints $c_{11}, c_{13}, c_{21}, c_{22} ,c_{32}, c_{33} < 0$ translate
into the six inequalities above. \qed
\end{ex}

To describe the combinatorial structure of the
eigenpair types, we introduce a simplicial complex $\Sigma_n$
on the vertex set $[n]^2$.  The facets (= maximal simplices) of $\Sigma_n$ are the
complements $[n]^2 \backslash \phi$ where $\phi$
runs over all connected functions on $[n]$.
Thus $\Sigma_n$ is pure of dimension $n^2-n-1$,
and the number of its facets equals (\ref{eq:oeisnumber}).
To each simplex $\sigma $ of $ \Sigma_n$ we associate
the simplicial cone $\R_{\geq 0} \{\bar e_{ij} : (i,j) \in \sigma \}$
in $\R^{n \times n}/V_n$.
We have shown that these cones form a decomposition of 
$\R^{n \times n}/V_n$ in the sense that every generic matrix lies
in exactly one cone. The last assertion in Theorem \ref{thm:zwei}
states that these cones do not form a fan.
We shall now show this for $n=3$
by giving a detailed combinatorial analysis of $\Sigma_3$.

\begin{ex}\label{ex:counter}[$n=3$] We here present the {\em proof of the third and final part of Theorem \ref{thm:zwei}}.
\\
The simplicial complex $\Sigma_3$ is $5$-dimensional, and it 
has $9$ vertices, $36$ edges, $81$ triangles,  etc.
The  f-vector  of $\Sigma_3$ is $(9,36, 81, 102, 66, 17) $.  The $17$ facets of
$\Sigma_3$ are, by definition, the set complements
of the $17$ connected functions $\phi$ on $[3] = \{1,2,3\}$. 
For instance, the connected function $\phi = \{12, 23, 31\}$ in 
Example \ref{lem:decomp} corresponds to the facet
$\{11,13,21,22,32,33\}$ of $\Sigma_3$. This $5$-simplex can be 
written as  $ \{11, 22, 33\}\ast\{21, 32, 13 \}$, the join of two triangles,
so it appears as  the central triangle on the left in Figure \ref{fig:complex}.

Figure \ref{fig:complex} is a pictorial representation of the simplicial
complex $\Sigma_3$. Each of the drawn graphs represents its clique complex,
and $*$ denotes the join of simplicial complexes.
The eight connected functions $\phi$ whose cycle has
length $\geq 2$ correspond to the eight facets on the left in Figure \ref{fig:complex}.
Here  the triangle $\{11,22,33\}$ is joined with the depicted cyclic triangulation of the boundary
of a triangular prism. The other nine facets of $\Sigma_3$ come in three
groups of three, corresponding to whether $1,2$ or $3$ is fixed by $\phi$.
For instance, if $\phi(3) = 3$ then the facet $[3]^2 \backslash \phi $ is the join of
the segment $\{11, 22\}$ with one of the three tetrahedra
in the triangulation of the solid triangular prism on the right in Figure~\ref{fig:complex}.

\begin{figure}
	\begin{center}
	 \includegraphics[width=\textwidth]{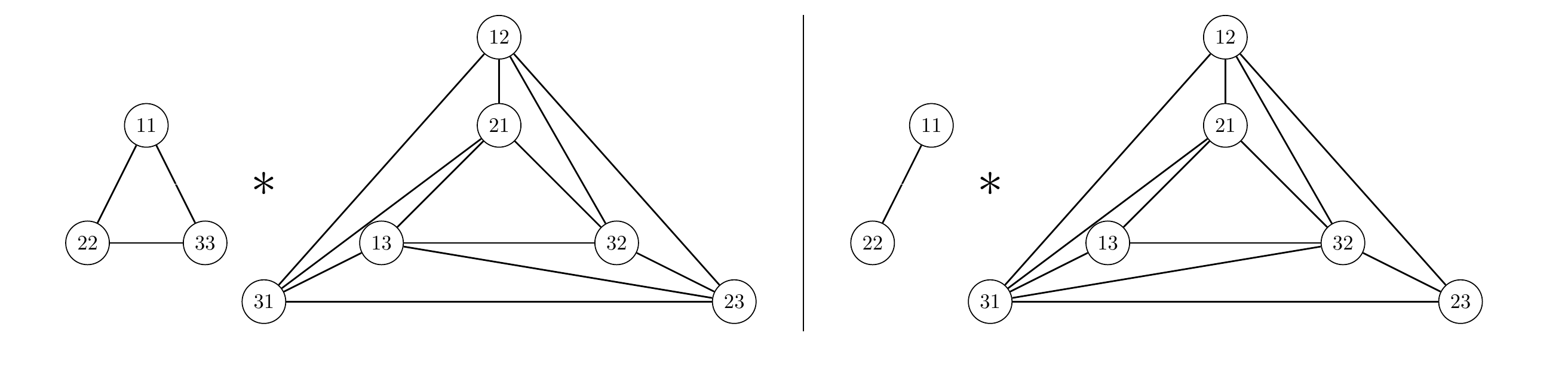}
	 \vskip -0.4cm
\caption{The simplicial complex $\Sigma_3$ of connected functions $\phi : [3] \rightarrow [3]$.
Fixed-point free $\phi$ are on the left and functions
with $\phi(3) = 3$ on the right.}
\label{fig:complex}
	\end{center}
\end{figure}

In the geometric realization given by the cones $\Omega_\phi$, the
square faces of the triangular prism are flat. However, we see that
both of their diagonals appear as edges in $\Sigma_3$.  This proves 
that the cones covering these diagonals do not fit together to form a fan.  \qed
\end{ex}

Naturally, each simplicial complex $\Sigma_n$  for $n > 3$ contains $\Sigma_3$
as a subcomplex, and this is compatible with the embedding of the cones.
Hence the eigenpair types fail to form a fan for any $n \geq 3$.
For the sake of concreteness,
we note that the $11$-dimensional simplicial complex $\Sigma_4$
has  f-vector $  (16, 120, 560, 1816, 4320, 7734, 10464, 10533, 7608, 3702, 1080, 142)$.

\medskip

The failure of the fan
property is caused by the existence of
matrices that have disjoint critical cycles. Such a matrix
lies in a lower-dimensional cone in the normal fan of
$\mathcal{C}_n$, and it has two or more
eigenvectors in $\mathbb{TP}^{n-1}$
that each arise from the unique eigenvectors
on neighboring full-dimensional cones.
These eigenvectors have distinct critical graphs $\phi$
and $\phi'$ and the cones $\Omega_\phi$
and $\Omega_{\phi'}$ do not intersect
along a common face.
In other words, the failure of the fan property 
reflects the discontinuity in the eigenvector map
$A \mapsto x(A)$.

For concrete example, 
consider the edge connecting $13$ and $23$
on the left in Figure \ref{fig:complex} and the
edge connecting $31$ and $32$ on the right in Figure \ref{fig:complex}.
These edges intersect in their relative interiors, thus violating the
fan property. In this intersection we find the matrix
\begin{equation}
\label{eq:onematrix}
 A \,\,\, = \,\,\, \begin{pmatrix}
\phantom{-} 0 & \phantom{-}0 & -1\,\, \\
\phantom{-} 0 & \phantom{-}0 & -1\,\,\\
-1 & -1 & \phantom{-} 0 \,\, 
\end{pmatrix},
\end{equation}
whose eigenspace is a tropical segment in $\mathbb{TP}^2$.
Any nearby generic matrix has a unique eigenvector, and that
eigenvector lies near one of the two endpoints of the tropical segment.
A diagram like Figure \ref{fig:complex} characterizes the combinatorial
structure of  such discontinuities.

\section{Skew-Symmetric Matrices}\label{sec:symm}

This project arose from the application of tropical eigenvectors
to the statistics problem of inferring rankings from pairwise comparison matrices.
This application was pioneered by Elsner and van den Driessche \cite{elsner, elsner10} and
further studied in \cite[\S 3]{nmt}.
Working on the additive scale, any pairwise comparison matrix $A = (a_{ij})$ is {\em skew-symmetric}, i.e. it satisfies
$a_{ij} + a_{ji} = 0$ for all $1 \leq i,j \leq n$.
The set $\wedge_2 \R^n$ of all skew-symmetric matrices is a linear subspace
of dimension $\binom{n}{2}$ in $\R^{n \times n}$. The input
of the {\em tropical ranking algorithm} is a generic matrix $A \in \wedge_2 \R^n $ and the
output is the permutation of $[n] = \{1,\ldots,n\}$ given by
sorting the entries of the eigenvector $x(A)$.
See \cite{nmt} for a comparison with other ranking methods.

In this section we are interested in the combinatorial types of eigenpairs when restricted to the space
$\wedge_2 \R^n$ of skew-symmetric matrices. In other words, we shall study the  decomposition of 
this space into the convex polyhedral cones $\Omega_\phi \cap \wedge_2 \R^n$
where $\phi$ runs over connected functions on $[n]$. Note that,
$\,\lambda(A) \geq 0 $ for all $A \in \wedge_2\R^n$, and the equality $\lambda(A) = 0$
holds if and only if $ A \in V_n$. Hence the intersection $\Omega_\phi \cap \wedge_2 \R^n$
 is trivial for all connected functions $\phi$ whose cycle has length $\leq 2$. This motivates the following definition.

We define a {\em kite} to be a connected function $\phi$ on $[n]$ whose cycle has length $\geq 3$.
By restricting the sum in (\ref{eq:oeisnumber}) accordingly, we see that
the number of kites on $[n]$ equals
\begin{equation}
\label{eq:oeisnumbers} \sum_{k=3}^n \frac{n!}{(n-k)!} \cdot n^{n-k-1}. 
\end{equation}
Thus the number of kites for $n=3,4,5,6,7,8$ equals
$ 2, 30, 444, 7320, 136590, 2873136 $.
The following result is the analogue to Theorem \ref{thm:zwei}
for skew-symmetric matrices.

\begin{thm} \label{thm:drei}
The open cones in $\wedge_2 \R^n$
on which the tropical eigenpair map
for skew-symmetric matrices is represented by distinct and unique linear functions
are $\Omega_\phi \cap \wedge_2 \R^n$ where $\phi$ runs over 
 all kites on $[n]$. Each cone has $n(n-3)$ facets, so it is not simplicial,
but it is  linearly isomorphic to $\R^{n-1}$ times the 
cone over the standard cube of dimension $\,n(n-3)/2 = \binom{n}{2}-n$.
This collection of cones does not form a fan for $n \geq 6$.
\end{thm}

\begin{proof}
It follows from our results in Section 2 that each cone of linearity 
of the map 
$$ \wedge_2 \R^n \rightarrow \R \times \mathbb{TP}^{n-1} \, , \,\,
A \mapsto (\lambda(A), x(A)) $$
has the form $\Omega_\phi \cap \wedge_2 \R^n$ for some kite $\phi$.
Conversely, let $\phi$ be any kite on $[n]$ with cycle $(1 \to 2 \to \ldots \to k \to 1)$.
We must show that  $\Omega_\phi \cap \wedge_2 \R^n$ has non-empty 
interior (inside $\wedge_2 \R^n$).
We shall prove the statement by induction on $n-k$. Note that this would prove distinctiveness, for the matrices constructed in the induction step lie strictly in the interior of each cones. The base case $n-k = 0$ is easy: here the skew-symmetric matrix
$A = \sum_{i=1}^n (e_{i \phi(i)} - e_{\phi(i) i})$ lies in the interior of $ \Omega_\phi$.

For the induction step, suppose that
$A $ lies in the interior of $\Omega_\phi \cap \wedge_2 \R^n$, and fix
an extension of $\phi$ to $[n+1]$ by setting $\phi(n+1) = 1$.
Our task is to construct a suitable matrix $A \in \wedge_2 \R^{n+1}$
that extends the old matrix and realizes the new $\phi$.
To do this, we need to solve for the $n$ unknown entries
$a_{i,n+1} = -a_{n+1,i}$, for $i=1,2,\ldots,n$.

By (\ref{eq:OmegaCone3}), 
 the necessary and sufficient conditions for $A$ to satisfy $\phi(n+1) = 1$ are
\begin{align*}
a_{(n+1)\, j} &\,\leq\, \lambda(A) + B(\phi_{(n+1) \, j^\ast}) - B(\phi_{jj^\ast}) ,\\
a_{j\, (n+1)} &\,\leq \, \lambda(A) + B(\phi_{j j^\ast}) - B(\phi_{1 j^\ast}) - B(\phi_{(n+1), 1}).
\end{align*}
Let $|\phi_{jj^\ast}|$ denote the number of edges in the path $\phi_{jj^\ast}$
Since $a_{ij} = -a_{ji}$, rearranging gives
\begin{align*}
a_{1\, (n+1)} + a_{(n+1) \, j}  \,&\,\leq\,\, A(\phi_{1j^\ast}) - A(\phi_{jj^\ast})  + (|\phi_{jj^\ast}|-|\phi_{1j^\ast}|)\lambda(A) ,\\
a_{1\, (n+1)} + a_{(n+1) \, j}\, &\,\geq \,\, A(\phi_{1j^\ast}) - A(\phi_{jj^\ast})  + (|\phi_{jj^\ast}|-|\phi_{1j^\ast}|)\lambda(A) - 2\lambda(A).
\end{align*}
The quantities on the right hand side are constants that do not depend on the
new matrix entries we seek to find. They specify a solvable system of upper
and lower bounds for the quantities $a_{1\, (n+1)} + a_{(n+1) \, j}$
 for $j = 2,\ldots,n$. Fixing these $n-1$ sums arbitrarily in their required intervals yields $n-1$ linear equations.
Working modulo the $1$-dimensional subspace
of $V_{n+1}$ spanned by $ \sum_{j=1}^n (e_{n+1,j} - e_{j,n+1}) $,
 we add the extra equation $\sum_{j=1}^n a_{j \, (n+1)} = 0$. From these $n$
linear equations, the missing matrix entries $a_{1(n+1)}, a_{2(n+1)}, \ldots, a_{n(n+1)}$ can be computed uniquely. The resulting matrix $A \in \wedge_2 \R^{n+1}$ strictly satisfies all the necessary
inequalities, so it is in the interior of the required cone
$\Omega_\phi$. 

The quotient of $\wedge_2 \R^{n}$ modulo its
$n$-dimensional subspace $V_n \cap \wedge_2 \R^{n}$ has dimension $n(n{-}3)/2$.
The cones we are interested in, one for each kite $\phi$,
 are all pointed in this quotient space. From the inductive construction above,
 we see that each cone $\Omega_\phi \cap \wedge_2 \R^n$
is characterized by upper and lower bounds on linearly independent linear forms.
This proves that this cone is linearly isomorphic to the cone over a
standard cube of dimension $n(n-3)/2$. 
If $n=4$ then the cubes are squares, as shown in
\cite[Figure 1]{nmt}.
 
Failure to form a fan stems from the existence of disjoint critical cycles,
as discussed at the end of Section 3. For $n \geq 6$, we can fix two disjoint triangles and their adjacent 
cones in the normal fan of $\mathcal{C}_n$. By an analogous  argument 
to that given in Example  \ref{ex:ineq},  we conclude that the cones 
$\Omega_\phi \cap \wedge_2 \R^n$, as $\phi$ runs over kites, do not form a fan for $n \geq 6$.  
\end{proof}

\smallskip

In this note we examined the division of the space of all (skew-symmetric)
$n {\times} n$-matrices into open polyhedral cones that represent
distinct combinatorial types of tropical eigenpairs. Since
that partition is not a fan, interesting phenomena
happen for special matrices $A$, i.e.~those not
 in any of the open cones $\Omega_\phi$.
For such matrices $A$, the eigenvalue $\lambda$ is still unique but the 
polyhedral set  $\,{\rm Eig}(A) = \{\, x \in \mathbb{TP}^{n-1} \,: \,A \odot x = \lambda \odot x \,\}\,$
may have dimension $\geq 1$. Let
 $B^* = B \oplus B^{\odot 2} \oplus \cdots \oplus B^{\odot n}$ and let
$B_0^*$ be the submatrix of $B^*$ given by all columns $i$ such that
$B^\ast_{ii} = 0$.  It is well known (see e.g.~\cite[\S 4.4]{butkovic}) that
$$ {\rm Eig}(A) \,=\, {\rm Eig}(B) \,=\, {\rm Eig}(B^*) \,=\, {\rm Image}(B_0^*). $$
Thus, ${\rm Eig}(A)$ is a tropical polytope
in the sense of Develin and Sturmfels \cite{DS}, and we 
 refer to ${\rm Eig}(A)$ as the {\em eigenpolytope} of the matrix $A$.
 This polytope has $\leq n$ tropical vertices.

Each tropical vertex of an eigenpolytope  ${\rm Eig}(A)$ 
can be represented as the limit of eigenvectors $x(A_\epsilon)$
where $(A_\epsilon)$ is a sequence of generic matrices 
lying in the cone $\Omega_\phi$ for some fixed
connected function $\phi$. This means that the
combinatorial structure of the eigenpolytope ${\rm Eig}(A)$ is
determined by the connected functions $\phi$ that are
adjacent to $A$. 

 For example, let us revisit  the 
 (inconsistently subdivided) square $\{13,  32, 23, 31\}$ in Figure~\ref{fig:complex}.
 The $3 \times 3$-matrices that correspond to the points on that square have the form
$$ A \quad = \quad \begin{pmatrix} 0 & 0 & a \\ 
                                                                0 & 0 & b \\
                                                                 c & d & 0 \end{pmatrix}, \quad
 \hbox{where $a,b,c,d < 0$}. 
$$
One particular instance of this was featured in (\ref{eq:onematrix}).
The eigenpolytope ${\rm Eig}(A)$ 
of the above matrix is the tropical line segment spanned by the columns of $A$,
and its two vertices are limits of the eigenvectors coming
from the two adjacent facets of $\Sigma_3$.

It would be worthwhile to study  this 
further in the skew-symmetric case.
Using kites, can one classify all
 tropical eigenpolytopes ${\rm Eig}(A)$
where $A $ ranges over matrices in $ \wedge_2  \R^n$? 

\medskip

\bibliographystyle{siam}

\medskip
\end{document}